\documentclass[11pt]{article}

\title{Sparse pancyclic subgraphs of random graphs}

\author{
Yahav Alon
\thanks{School of Mathematical Sciences, Raymond and Beverly Sackler Faculty of Exact Sciences, Tel Aviv University,
Tel Aviv, 6997801, Israel. Email: yahavalo@tauex.tau.ac.il.}
\and Michael Krivelevich
\thanks{School of Mathematical Sciences, Raymond and Beverly
Sackler Faculty of Exact Sciences, Tel Aviv University, Tel Aviv,
6997801, Israel. Email: krivelev@tauex.tau.ac.il. Partially supported by USA-Israel BSF grant 2018267.}
}

\usepackage{tikz}
\usepackage{mathtools}
\usepackage{verbatim}
\usetikzlibrary{arrows,%
                petri,%
                topaths,%
                calc,%
                decorations.pathmorphing,%
                decorations.shapes}%
\usepackage{tkz-berge}
\usepackage[position=top]{subfig}
\usepackage{amsmath,amsthm, amssymb,latexsym, amsfonts}
\usepackage{algorithm}
\usepackage{enumitem}

\begin{document}
\maketitle
\newtheorem{thm}{Theorem}
\newtheorem{theorem}{Theorem}[section]
\newtheorem{propos}{Proposition}
\newtheorem{defin}{Definition}
\newtheorem{lemma}{Lemma}[section]
\newtheorem{corol}{Corollary}[section]
\newtheorem{corol*}{Corollary}
\newtheorem{thmtool}{Theorem}[section]
\newtheorem{corollary}[thmtool]{Corollary}
\newtheorem{lem}[thmtool]{Lemma}
\newtheorem{defi}[thmtool]{Definition}
\newtheorem{prop}[thmtool]{Proposition}
\newtheorem{clm}[thmtool]{Claim}
\newtheorem{claim}[thmtool]{Claim}
\newtheorem{conjecture}{Conjecture}
\newtheorem{problem}{Problem}
\newcommand{\Proof}{\noindent{\bf Proof.}\ \ }
\newcommand{\Remarks}{\noindent{\bf Remarks:}\ \ }
\newcommand{\Remark}{\noindent{\bf Remark:}\ \ }

\newcommand{\Dist}[1]{\mathsf{#1}}
\newcommand{\Bin}{\Dist{Bin}}
\newcommand{\HH}{\mathcal{H}}
\newcommand{\LL}{\mathcal{L}}
\newcommand{\pr}{\mathbb{P}}
\newcommand{\sm}{\text{SMALL}}
\newcommand{\cl}{\text{CLOSE}}
\newcommand{\bd}{\text{BAD}}
\newcommand{\lrg}{\text{LARGE}}
\newcommand{\enm}{\text{END}_M}
\newcommand{\tend}{\text{END}}
\newcommand{\lmax}{L_{\max}}
\newcommand{\kout}{G_{k\text{-out}}}
\newcommand{\dkout}{D_{k\text{-out}}}
\newcommand{\E}{\mathbb{E}}
\newcommand{\Var}{\text{Var}}
\newcommand{\cA}{\mathcal{A}}
\newcommand{\pex}{\text{Pex}}

\newcommand{\YA}[1]{{\color{red} \small{(YA: #1)}}}

\begin{abstract}
It is known that the complete graph $K_n$ contains a pancyclic subgraph with $n+(1+o(1))\cdot \log _2 n$ edges, and that there is no pancyclic graph on $n$ vertices with fewer than $n+\log _2 (n-1) -1$ edges. We show that, with high probability, $G(n,p)$ contains a pancyclic subgraph with $n+(1+o(1))\log_2 n$ edges for $p \ge p^*$, where $p^*=(1+o(1))\ln n/n$, right above the threshold for pancyclicity,
\end{abstract}

\section{Introduction} \label{sec-intro}

Say that a graph $G$ is pancyclic if $G$ contains a cycle of every length between 3 and $|V(G)|$. See monograph \cite{PANBOOK} for generic information on pancyclic graphs. In his influential paper on pancyclic graphs, Bondy \cite{BOND} asked what is the minimum number of edges in a pancyclic $n$-vertex graph. This can be rephrased as the minimum number of edges in a pancyclic subgraph of $K_n$, which motivates the following definition.

\begin{defin}
Say that a pancyclic graph $G$ on $n$ vertices has \textit{pancyclicity excess} $k$, and denote $\pex (G)=k$, if the minimum number of edges in a pancyclic subgraph of $G$ is $n+k$.
\end{defin}

In other words, a pancyclic subgraph of $G$ achieving the minimum number of edges is formed by a Hamilton cycle and $\pex (G)$ additional chords. In his paper, Bondy stated that, for every $n$,
$$
\log _2 (n-1) -1 \le \pex (K_n) \le \log _2 n + \log ^* n +O(1),
$$
and did not provide a proof. Shi \cite{SHI} later asserted the lower bound, by showing that an $n$-vertex graph with $n+k$ edges contains at most $2^{k+1}-1$ distinct cycles, so every subgraph of $K_n$ with fewer than $n+\log _2(n-1)-1$ edges must have fewer than $2^{\log _2(n-1)}-1=n-2$ cycles in total, regardless of their lengths. On the other hand, there are constructions for every $n$ of an $n$-vertex pancyclic graph with $\log _2 n + \log ^* n +O(1)$ chords (see e.g. \cite{PANBOOK}, Chapter 4.5), so $\pex (K_n) \le \log _2 n + \log ^* n +O(1)$. What is the exact value of $\pex (K_n)$ within this range is still an open question.

In this paper, we study the typical behaviour of $\pex (G)$, for $G\sim G(n,p)$. Cooper and Frieze \cite{CF} showed that, for $p\in [0,1]$, the limiting probability of $G\sim G(n,p)$ being pancyclic is
\begin{eqnarray*}
\lim _{n\to \infty} \pr (G(n,p)\text{ is pancyclic}) =
	\begin{cases}
        1 & \text{if } np-\log n-\log \log n \to \infty ;\\
        e^{-e^{-c}} & \text{if } np-\log n-\log \log n \to c;\\
        0 & \text{if } np-\log n-\log \log n \to -\infty .
    \end{cases}
\end{eqnarray*}
Here and later, if the base of the logarithm is not stated then it is the natural base. The above expression is also the limiting probability of $G$ being Hamiltonian, and the limiting probability of $\delta (G) \ge 2$. In particular, the three properties have the same threshold.

Clearly, $\pex (G) \ge \log _2 (n-1) -1$ for every pancyclic graph $G$ on $n$ vertices. On the other hand, Cooper \cite{COOPER1} showed that if $p$ is above the pancyclicity threshold, then with high probability $G\sim G(n,p)$ is a so called \textit{1-pancyclic} graph, that is, it contains a Hamilton cycle $H$ with the property that, for every $\ell \in [3,n-1]$, there is an edge $e\in E(G)$ such that $H\cup \{e\}$ contains a cycle of length $\ell$ and a cycle of length $n-\ell +2$. Observe that if $G$ is a 1-pancyclic $n$-vertex graph then $\pex (G) \le \lceil \frac{n-3}{2} \rceil$. So Cooper's result implies that $\pex (G(n,p)) \le \lceil \frac{n-3}{2} \rceil$ with high probability, for all $p$ above the pancyclicity threshold.

Our result in this paper shows that, for $G\sim G(n,p)$, the pancyclicity excess of $G$ is, typically, very close to the above stated general lower bound.

\begin{thm}\label{thm:main}
There is $p^*=p^*(n)=(1+\varepsilon (n))\cdot \frac{\log n}{n}$, where $\varepsilon (n) = O\left( \frac{1}{\log \log} \right)$, such that, if $p\ge p^*$ and $G\sim G(n,p)$, then with high probability $g$ is pancyclic with $\pex (G) = (1+o(1))\cdot \log _2 n$.
\end{thm}

It is worth noting that we did not attempt to optimize the error term $\varepsilon (n)$, opting rather for a more simple proof. We therefore leave the question of whether $\pex (G(n,p))$ also typically satisfies $\pex (G(n,p))=(1+o(1))\cdot \log _2n$ for all $p$ above the pancyclicity threshold as an open question. 

\paragraph*{Paper structure} In Section \ref{sec-per} we introduce definitions and notation required for the rest of the paper, as well as auxiliary results to be used in our proof. In Section \ref{sec:outline} we introduce a construction of a subgraph of a given $n$-vertex graph, which, if successful, produces a subgraph with $n+(1+o(1))\cdot \log _2n$ edges. In Section \ref{sec:construction} we show that, with high probability, the construction is possible in $G(n,p)$ for $p\ge p^*$, and in Section \ref{sec:main} we complete the proof of Theorem \ref{thm:main} by showing that the constructed subgraph is pancyclic.

\section{Preliminaries} \label{sec-per}

\subsection{Definitions and notation}

The following graph theoretic notation is used throughout the paper.

Let $G$ be a graph and $U,W\subseteq V(G)$ vertex subsets. We denote by $E_G(U,W)$ the set of edges of $G$ with vertex in $U$ and one vertex in $W$, and $e_G(U,W)=|E_G(U,W)|$.
We let $G[U]$ denote the subgraph induced by $G$ on the vertex subset $U$, by $E_G(U)$ the set of edges in $G[U]$, and by $e_G(U)$ its size.
We denote by $N_G(U)$ the (external) neighbourhood of $U$, that is, the set of vertices in $V(G)\setminus U$ adjacent to a vertex of $U$.
The degree of a vertex $v\in V(G)$, denoted by $d_G(v)$, is the number of edges of $G$ incident to $v$.

We let $\mathcal{L}(G)$ denote the set of cycle lengths found in $G$, that is, $\mathcal{L}(G)$ is the set of integers $\ell$ such that $G$ contains a cycle of length $\ell$.

While using the above notation we occasionally omit $G$ if the identity of the specific graph is clear from the context.

We occasionally suppress the rounding signs to simplify the presentation.

Finally, we require the following definition.

\begin{defin}
A graph $G$ is called a $(k,\alpha)$-expander if every subset $U\subseteq V(G)$ with $|U|\le k$ satisfies $|N_G(U)| \ge \alpha \cdot |U|$.
\end{defin}

\subsection{Auxiliary results}

\begin{theorem}[Cycle lengths in $G(n,p)$, a corollary of {\L}uczak \cite{LUC}]\label{thm:luczak}
Let  $p=p(n)$ be such that $np\to \infty$ and let $G\sim G(n,p)$. Then, with high probability, $\left[ 3,0.99n \right] \subseteq \LL (G)$.
\end{theorem}

\begin{theorem}[Tree embeddings in expanders, a corollary of \cite{HAX} as given in \cite{BALETAL}]\label{thm:hax}
Let $N,\Delta $ be integers, and let $G$ be a graph. Assume that there exists an integer $k$ such that
\begin{enumerate}
\item For every $U\subseteq V(G)$ with $|U| \le k$ we have $|N_G(U)|\ge \Delta \cdot |U|+1$;
\item For every $U\subseteq V(G)$ with $k<|U| \le 2k$ we have $|N_G(U)|\ge \Delta \cdot |U|+N$.
\end{enumerate}
Then, for every $v\in V(G)$ and every rooted tree $T$ with at most $N$ vertices and maximum degree at most $\Delta$, the graph $G$ contains a copy of $T$ rooted in $v$.
\end{theorem}

\begin{lemma}[Hamiltonicity and expansion of $G(n,p)$, see e.g. \cite{KRIVI}, Section 4]\label{thm:kriv}
Let $p=p(n)$ be such that $np-\log n-\log \log n \to \infty$, and let $G\sim G(n,p)$. Then, with high probability, there is a subset $S\subseteq V(G)$ of $\frac{n}{4}$ vertices, such that for every $s\in S$ there is a subset $T_s\subseteq V(G)$ of $\frac{n}{4}$ vertices, and for every $t\in T_s$ there is a Hamilton path between $s$ and $t$.
\end{lemma}

\section{The constructed pancyclic subgraph} \label{sec:outline}

We emulate (an approximation of) the construction in \cite{PANBOOK}.

\begin{defin}
Let $G$ be a graph and $H\subseteq G$ be a Hamilton cycle, and let $2\le \ell \le n-2$. We say that an edge $e\in E(G)$ is an $\ell$\emph{-shortcut} with respect to $H$ if (at least) one of the two intervals on $H$ that connects the two endpoints of $e$ has length $\ell +1$.
\end{defin}

The motivation behind this definition is that by using $H$ and an $\ell$-shortcut we can find a cycle of length $n-\ell$ in $G$, by replacing an interval of length $\ell +1$ with a single edge (the $\ell$-shortcut). In the construction described in \cite{PANBOOK}, one creates a sparse pancyclic graph by taking an $n$-cycle $H$ and $K$ shortcuts $e_0,e_1,...,e_K$, where $K$ is such that $\frac{1}{2}n \le 2^{K+1} +K-1 \le n$ and $e_i$ is a $2^i$-shortcut. Additionally, these shortcuts are consecutive on the cycle, so that $e_i,e_{i+1}$ and their corresponding intervals intersect in a vertex $v_i$. By taking intervals from the cycle $H$ and a subset of shortcuts, one can now encode a cycle of every length between $n$ and $n-2^{K+1}+1$. Next, by adding the edge between the first vertex of $e_0$ and the second vertex of $e_K$, all cycle lengths between $K+2$ and $2^{K+1}+K$ can be encoded. This leaves out only a subset of cycle lengths contained in $[5,K+1]$, and adding these lengths to the set of cycle lengths in the graph can be done by inserting $O(\log ^*n)$ additional edges. For the full details of the construction, we refer the reader to \cite{PANBOOK} Chapter 4.5.

We approximate this construction by finding a Hamilton cycle and shortcuts to encode an interval of $L =\Omega \left( \frac{n}{\sqrt{\log n}} \right)$ consecutive cycle lengths. Like in the deterministic version, we will utilize binary encoding of the cycle lengths, so that the number of required shortcuts is $(1+o(1))\log _2n$. Additionally, we will require the shortcuts to reside on a short interval of the cycle (where in the deterministic version they intersected each other in a vertex). Next, by adding certain edges to the subgraph we can add an interval of $L$ cycle lengths with each such added edge. If the said additional edges are chosen well (which we will show is possible to do with high probability), one can get a union of $O(\sqrt{\log n})$ of these intervals that covers all the lengths between some initial length $\ell ^*=(1+o(1)) \log _2n$ and $n$.

To handle cycle lengths shorter than $\ell ^*$ we will show that, with high probability, almost all of them (that is, all but $o(\log n)$ cycle lengths in $[3, \ell ^*]$) can be encoded by $o(\log n)$ carefully chosen shortcuts, this time utilizing an encoding in base $b=\lceil \log \log n \rceil$. The remaining unencoded cycle lengths, which constitute a subset of $[3,\ell ^*]$ of size $o(\log n)$, can now be added one-by-one by using at most $o(\log n)$ additional edges, with high probability.

Let
$$
p_1=p_5 = \frac{2\log \log n}{n},\ p_2=p_3 = \frac{50\log n}{n\cdot \log \log n},\  p_4=\frac{\log n + 10\sqrt{\log n}}{n},
$$
and let
$$
p^* =p^*(n)= 1-\prod _{i=1}^5(1-p_i).
$$
Letting $\varepsilon (n) \coloneqq \frac{n}{\log n}\cdot p^*-1$ we get that $\varepsilon (n) = O(\frac{1}{\log \log n})$, and since the property $\pex (G)\le k$ is monotone increasing, it suffices to prove that $\pex (G) \le (1+o(1))\log _2n$ holds with high probability for $G(n,p^*) \sim \bigcup _{i=1}^5G(n,p_i)$. We note that we did not attempt to optimize the value of $\varepsilon (n)$ determined by $p_1,...,p_5$, aiming rather for simplicity.

Denote
$$
\ell _i \coloneqq 2^i+1,
$$
and
$$
\beta = \beta (n) \coloneqq \frac{2(\log \log n) ^2}{\log n},\ d =d(n) = \lfloor \log _{(5\beta)^{-1}}(n/200) \rfloor .
$$
Note that
$$
d = \lfloor \log _{(5\beta)^{-1}}(n/200) \rfloor = (1+o(1))\cdot \frac{\log (n/200)}{-\log (5\beta )} = (1+o(1))\cdot \frac{\log n}{\log \log n}.
$$
For $1\le i \le 5$ let $G_i \sim G(n,p_i)$. We divide the construction into five steps, where in the $i$'th step we sample $G_i$ to try and produce a subgraph $H_i\subseteq \bigcup _{j=1}^iG_j$. If the construction is successful, the produced subgraph $H_5$ will be pancyclic with $|E(H_5)| = n+(1+o(1))\cdot \log_2n$. The steps of our construction are as follows.

\begin{enumerate}

\item Let
$$
K_0 \coloneqq \lfloor \log _2 \left( \frac{\log n}{6\log \log \log n} \right) \rfloor ,
$$
and
$$
b \coloneqq \lceil \log \log n \rceil,\ t\coloneqq \lceil \log _b \log n \rceil.
$$
Find a set of vertex disjoint cycles $C_0,...,C_{K_0},C_{\text{short}}$ in $G_1$ of respective lengths $\ell _0+1,\ell_1+1,...,\ell_{K_0}+1,t\cdot b+1$. The first $K_0+1$ cycles will later become the first $K_0+1$ shortcuts, and their corresponding intervals, where the edges of $C_{\text{short}}$ will become the shortcuts required to handle short cycles. For every $0\le i \le K_0$, choose an arbitrary edge $e_i\in C_i$ to serve as the shortcut.  Denote $H_1 = C_{\text{short}} \cup \bigcup _{i=0}^{K_0}C_i$.

\begin{figure}[h]
\centering
\begin{tikzpicture}[vertex/.style={draw,circle,color=black,fill=black,inner sep=1,minimum width=4pt},scale=1]
    
    \node[vertex] (a1) at (-0.3,0) {};
    \node[vertex] (a2) at (0,0.5) {};
    \node[vertex] (a3) at (0.3,0) {};
    \draw[thick] (a1) to (a2);
    \draw[thick] (a3) to (a2);
    \draw[thick] (a1) to (a3);
	\node at (0,-0.4) {$\ell _0=2$};
	
	\node[vertex] (b1) at (1.2,0) {};
    \node[vertex] (b2) at (1.2,0.6) {};
    \node[vertex] (b3) at (1.8,0.6) {};
    \node[vertex] (b4) at (1.8,0) {};
    \draw[thick] (b1) to (b2);
    \draw[thick] (b3) to (b2);
    \draw[thick] (b4) to (b3);
    \draw[thick] (b4) to (b1);
	\node at (1.5,-0.4) {$\ell _1 =3$};
	
	\node[vertex] (c1) at (2.9,0) {};
    \node[vertex] (c2) at (2.7,0.5) {};
    \node[vertex] (c3) at (2.9,1) {};
    \node[vertex] (c4) at (3.5,1) {};
    \node[vertex] (c5) at (3.7,0.5) {};
    \node[vertex] (c6) at (3.5,0) {};
    \draw[thick] (c1) to (c2);
    \draw[thick] (c3) to (c2);
    \draw[thick] (c4) to (c3);
    \draw[thick] (c4) to (c5);
    \draw[thick] (c6) to (c5);
    \draw[thick] (c6) to (c1);
	\node at (3.2,-0.4) {$\ell _2 =5$};
	
	\node at (4.5,0.3) {$\cdots$};
	
	\node[vertex] (d1) at (7,0.1) {};
    \node[vertex] (d2) at (7.6,0.1) {};
    \node[vertex] (d3) at (6.4,0.1) {};
    \node[vertex] (d4) at (8.1,0.2) {};
    \node[vertex] (d5) at (5.9,0.2) {};
    \node[vertex] (d6) at (8.5,0.5) {};
    \node[vertex] (d7) at (5.5,0.5) {};
    \node[vertex] (d8) at (8.8,1) {};
    \node[vertex] (d9) at (5.2,1) {};
    \draw[thick] (d1) to (d2);
    \draw[thick] (d3) to (d1);
    \draw[thick] (d4) to (d2);
    \draw[thick] (d3) to (d5);
    \draw[thick] (d6) to (d4);
    \draw[thick] (d5) to (d7);
    \draw[thick] (d6) to (d8);
    \draw[thick] (d9) to (d7);
    \draw[densely dotted,thick] (d8) to (8.9,1.5);
    \draw[densely dotted,thick] (d9) to (5.1,1.5);
	\node at (7,-0.4) {$\ell _{K_0} =\Theta \left( \frac{\log n}{\log \log \log n} \right)$};

	\node[vertex] (e1) at (10.5,0.1) {};
    \node[vertex] (e2) at (11.1,0.1) {};
    \node[vertex] (e3) at (10,0.4) {};
    \node[vertex] (e4) at (11.6,0.4) {};
    \node[vertex] (e5) at (9.7,0.9) {};
    \node[vertex] (e6) at (11.9,0.9) {};
    \draw[thick] (e5) to (e3) to (e1) to (e2) to (e4) to (e6);
    \draw[densely dotted,thick] (e5) to (9.7,1.5);
    \draw[densely dotted,thick] (e6) to (11.9,1.5);
	\node at (11,-0.4) {$t\cdot b =\Theta \left( \frac{\log \log n}{\log \log \log n} \right)$};

\end{tikzpicture}
\caption{Step 1, with resulting graph $H_1$ depicted.}\label{fig:stage1}
\end{figure}
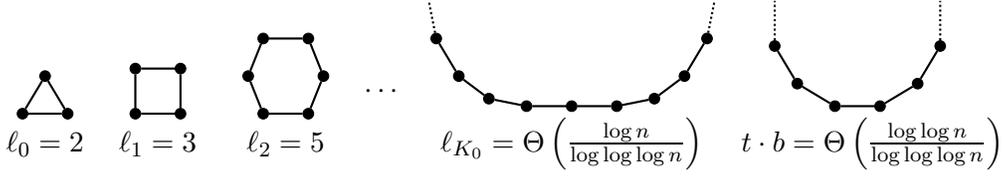

\item For every $0\le i \le K_0$, find a path of length $d+2$ in $G_2$ between the second vertex of $e_i$ and the first vertex of $e_{i+1}$ (where for $i=K_0$ the path is between $e_{K_0}$ and $e_0$), so that the $K_0+1$ paths are pairwise vertex disjoint from each other, and internally vertex disjoint from $V(H_1)$. Call the cycle formed by the union of the paths and the shortcuts $C^*$ and denote $\ell ^* \coloneqq e(C^*), H_2\coloneqq H_1 \cup C^*$. We have
$$
\ell ^* = (1+o(1))\cdot K_0 \cdot d=(1+o(1))\cdot \log _2n.
$$

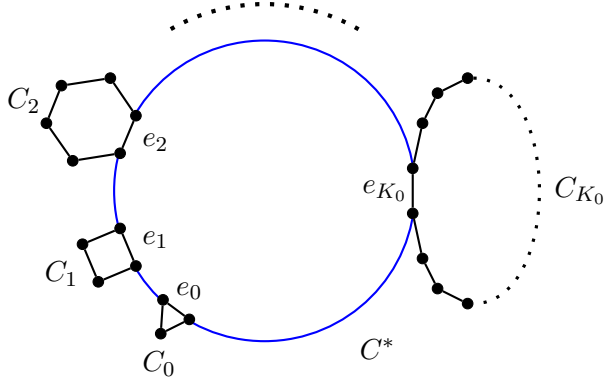
\begin{figure}[h]
\centering
\begin{tikzpicture}[vertex/.style={draw,circle,color=black,fill=black,inner sep=1,minimum width=4pt},scale=1]
    
    \draw[thick,blue] (0,0) circle (2);
    \node at (1.5,-2.1) {$C^*$};
	
	\fill[white] (-1.175,-1.58) circle (0.25);
	\node[vertex] (a1) at (-1,-1.71) {};
    \node[vertex] (a3) at (-1.35,-1.45) {};
    \node[vertex] (a2) at (-1.38,-1.9) {};
    \draw[thick] (a1) to (a2);
    \draw[thick] (a3) to (a2);
    \draw[thick] (a1) to (a3);
	\node at (-1.4,-2.3) {$C _0$};
	\node at (-1,-1.33) {$e _0$};
	
	\fill[white] (-1.81,-0.75) circle (0.25);
	\node[vertex] (b1) at (-1.71,-1) {};
    \node[vertex] (b2) at (-2.21,-1.21) {};
    \node[vertex] (b3) at (-2.42,-0.71) {};
    \node[vertex] (b4) at (-1.92,-0.5) {};
    \draw[thick] (b1) to (b2);
    \draw[thick] (b3) to (b2);
    \draw[thick] (b4) to (b3);
    \draw[thick] (b4) to (b1);
	\node at (-2.7,-1.1) {$C _1$};
	\node at (-1.45,-0.65) {$e _1$};
	
	\fill[white] (-1.81,0.75) circle (0.25);
	\node[vertex] (c1) at (-1.71,1) {};
    \node[vertex] (c2) at (-1.92,0.5) {};
    \node[vertex] (c5) at (-2.7,1.4) {};
    \node[vertex] (c4) at (-2.9,0.9) {};
    \node[vertex] (c3) at (-2.55,0.4) {};
    \node[vertex] (c6) at (-2.05,1.5) {};
    \draw[thick] (c1) to (c2);
    \draw[thick] (c3) to (c2);
    \draw[thick] (c4) to (c3);
    \draw[thick] (c4) to (c5);
    \draw[thick] (c6) to (c5);
    \draw[thick] (c6) to (c1);
	\node at (-3.2,1.2) {$C _2$};
	\node at (-1.45,0.65) {$e _2$};
	
	\draw[ultra thick,loosely dotted] (1.25,1.25*1.72) arc (60:120:2.5);
	
	\fill[white] (2,0) circle (0.3);
	\node[vertex] (d1) at (1.97,0.3) {};
    \node[vertex] (d2) at (2.1,0.9) {};
    \node[vertex] (d3) at (2.3,1.3) {};
    \node[vertex] (d4) at (2.7,1.5) {};
    
    \node[vertex] (d5) at (1.97,-0.3) {};
    \node[vertex] (d6) at (2.1,-0.9) {};
    \node[vertex] (d7) at (2.3,-1.3) {};
    \node[vertex] (d8) at (2.7,-1.5) {};
    
    \draw[thick] (d1) to (d2);
    \draw[thick] (d1) to (d5);
    \draw[thick] (d3) to (d2);
    \draw[thick] (d4) to (d3);
    \draw[thick] (d5) to (d6);
    \draw[thick] (d6) to (d7);
    \draw[thick] (d7) to (d8);
    \draw[very thick,loosely dotted,out=0,in=0] (d4) to (d8);
    
	\node at (1.6,0) {$e _{K_0}$};
	\node at (4.2,0) {$C _{K_0}$};	
	
\end{tikzpicture}
\caption{Step 2, with resulting graph $H_2$ depicted.}\label{fig:stage2}
\end{figure}

\item Let 
\begin{align*}
K & \coloneqq \lfloor \log _2 \left( \frac{n}{\sqrt{\log n}} \right) \rfloor ,\\
L & \coloneqq 2^{K+1}-1 ,
\end{align*}
so that $L+1\in \left[ \frac{n}{\sqrt{\log n}},\frac{2n}{\sqrt{\log n}} \right]$. For $K_0 < i \le K$, construct the $i$'th shortcut by choosing an arbitrary (non-shortcut) edge $e_i$ on $C^*$, and finding a path of length $\ell _i$ between its two vertices in $G_3$, such that these paths are internally vertex disjoint from each other and from $V(H_2)$. Letting $C_i$ denote the cycle comprised of $e_i$ and the $\ell_i$-path in $G_3$ between its vertices, we get that the subgraph $C^* \cup \bigcup _{i=0}^KC_i$ contains all cycle lengths in the interval $\left[ \ell ^*, \ell^*+L \right]$. Choose an arbitrary edge $e^* \in C^* \setminus \{ e_0 ,...,e_{K}\}$.

Next, denote $E (C_{\text{short}}) = \{e_{\text{short}}\} \cup \{ e_{i,j} \mid 0\le i \le t-1,0\le j \le b-1 \}$, with an arbitrary order. Find paths $\{ P_{i,j} \mid 0\le i \le t-1,0\le j \le b-1 \}$ in $G_3$, where $P_{i,j}$ connects the endpoints of $e_{i,j}$, such that the paths are all internally vertex disjoint from each other, and from $V(H_2\cup \bigcup _{i={K_0}}^KC_i)$, and $P_{i,j}$ has length $d+2+j\cdot b^i$. Now the subgraph $C_{\text{short}}\cup \{ P_{i,j} \mid 0\le i \le t-1,0\le j \le b-1 \}$ contains all cycle lengths in $\left [ (d+b+1)\cdot t+1, (d+b+1)\cdot t + b^t \right]$. Note that $b^t \ge \log n > \ell ^*$, and that $(d+b+1)\cdot t = O\left( \frac{\log n}{\log \log \log n} \right)$.

Finally for this step, connect one vertex of $e^*$ to one vertex of $e_{\text{short}}$ by a path $P^*$ of length $d+2$, internally disjoint from all previous construction, and denote $H_3 \coloneqq H_2 \cup P^* \cup \{ P_{i,j} \}_{i,j}  \cup \bigcup _{i={K_0}}^KC_i.$

\begin{figure}[h]
\centering
\begin{tikzpicture}[vertex/.style={draw,circle,color=black,fill=black,inner sep=1,minimum width=4pt},scale=1]
    
    \draw[thick] (0,0) circle (1.6);
    \node at (0,0) {$C^*$};
    \node at (-1.25,-0.1) {$e^*$};
    
    \node[vertex] (path1) at (180:1.6) {};
	
	\node[vertex] (a1) at (240:1.6) {};
    \node[vertex] (a2) at (230:1.6) {};   
    \draw[out=240,in=230] (a1) to (a2); 	
    
	\node[vertex] (b1) at (200:1.6) {};
    \node[vertex] (b2) at (190:1.6) {};  
    \draw[out=200,in=285] (b1) to (195:1.8)
    [out=105,in=190] (195:1.8) to (b2); 	
    
	\node[vertex] (c1) at (135:1.6) {};
    \node[vertex] (c2) at (125:1.6) {};   
    \draw[out=135,in=220] (c1) to (130:1.9)
    [out=40,in=125] (130:1.9) to (c2); 		
    
	\node[vertex] (d1) at (60:1.6) {};
    \node[vertex] (d2) at (50:1.6) {};   
	\draw[out=60,in=145] (d1) to (55:2.03)
    [out=325,in=50] (55:2.03) to (d2); 
    
    	\node[vertex] (e1) at (345:1.6) {};
    \node[vertex] (e2) at (335:1.6) {};   
	\draw[out=345,in=70] (e1) to (340:2.1)
    [out=250,in=335] (340:2.1) to (e2); 
    
	\node[vertex] (f1) at (220:1.6) {};
    \node[vertex] (f2) at (210:1.6) {};  
    \draw[out=220,in=305,blue] (f1) to (215:2.45)
    [out=125,in=210,blue] (215:2.45) to (f2);
    \draw[out=230,in=315,blue] (a2) to (225:2.6)
    [out=135,in=220,blue] (225:2.6) to (f1); 
    
    	\node[vertex] (g1) at (160:1.6) {};
    \node[vertex] (g2) at (150:1.6) {};  
    \draw[out=160,in=245,blue] (g1) to (155:2.35)
    [out=65,in=150,blue] (155:2.35) to (g2);	
    
    	\node[vertex] (h1) at (107:1.6) {};
    \node[vertex] (h2) at (97:1.6) {};  
    \draw[out=107,in=192,blue] (h1) to (102:2.25)
    [out=12,in=97,blue] (102:2.25) to (h2);	
    
    	\node[vertex] (i1) at (85:1.6) {};
    \node[vertex] (i2) at (75:1.6) {};  
    \draw[out=85,in=170,blue] (i1) to (80:2.5)
    [out=350,in=75,blue] (80:2.5) to (i2);
    \draw[out=92,in=181,blue] (h2) to (91:2.8)
    [out=1,in=88,blue] (91:2.8) to (i1);
    
    	\node[vertex] (j1) at (20:1.6) {};
    \node[vertex] (j2) at (10:1.6) {};  
    \draw[out=20,in=105,blue] (j1) to (15:2.35)
    [out=285,in=10,blue] (15:2.35) to (j2);	
	
    	\node[vertex] (k1) at (320:1.6) {};
    \node[vertex] (k2) at (310:1.6) {};  
    \draw[out=320,in=45,blue] (k1) to (315:2.25)
    [out=225,in=310,blue] (315:2.25) to (k2);
    
    	\node[vertex] (l1) at (290:1.6) {};
    \node[vertex] (l2) at (280:1.6) {};  
    \draw[out=290,in=15,blue] (l1) to (285:2.6)
    [out=195,in=280,blue] (285:2.6) to (l2);
    
    	\node[vertex] (m1) at (270:1.6) {};
    \node[vertex] (m2) at (260:1.6) {};  
    \draw[out=270,in=355,blue] (m1) to (265:2.4)
    [out=175,in=260,blue] (265:2.4) to (m2); 
    
    	\node[vertex] (n1) at (30:1.6) {};
    \node[vertex] (n2) at (20:1.6) {};  
    \draw[out=20,in=115,blue] (n1) to (25:2.8)
    [out=295,in=30,blue] (25:2.8) to (n2);

    \coordinate (mid) at (-7,0);
    \node at (mid) {$C_{\text{short}}$};
    \draw[thick] (mid) circle (0.7);
    \node at (-5.8,-0.3) {$e_{\text{short}}$};
    \node at (-4,0.7) {$P^*$};
    
    \node[vertex] (path1) at (180:1.6) {};
    \node[vertex] (path2) at ($ (mid) +(0:0.7)$ ) {};
    \draw[thick,blue,decorate,decoration={snake,segment length=41,amplitude=10}] (path1) to (path2); 	
    	
    	\node[vertex] (x1) at ($(mid)+(30:0.7)$) {};
    	\node[vertex] (x2) at ($(mid)+(60:0.7)$) {};
    	\node[vertex] (x3) at ($(mid)+(90:0.7)$) {};
    	\node[vertex] (x4) at ($(mid)+(120:0.7)$) {};
    	\node[vertex] (x5) at ($(mid)+(150:0.7)$) {};
    	\node[vertex] (x6) at ($(mid)+(180:0.7)$) {};
    	\node[vertex] (x7) at ($(mid)+(210:0.7)$) {};
    	\node[vertex] (x8) at ($(mid)+(240:0.7)$) {};
    	\node[vertex] (x9) at ($(mid)+(270:0.7)$) {};
    	\node[vertex] (x10) at ($(mid)+(300:0.7)$) {};
    	\node[vertex] (x11) at ($(mid)+(330:0.7)$) {};
    	
    \draw[out=45,in=0,blue] (path2) to (x1);	
    \draw[out=30,in=315,blue] (x1) to ($(mid)+(45:0.87)$)
    [out=135,in=60,blue] ($(mid)+(45:0.87)$) to (x2);	
    \draw[out=60,in=345,blue] (x2) to ($(mid)+(75:0.95)$)
    [out=165,in=90,blue] ($(mid)+(75:0.95)$) to (x3);
    
    \draw[out=135,in=90,blue] (x3) to (x4);	
    \draw[out=120,in=45,blue] (x4) to ($(mid)+(135:1)$)
    [out=225,in=150,blue] ($(mid)+(135:1)$) to (x5);	
    \draw[out=150,in=75,blue] (x5) to ($(mid)+(165:1.15)$)
    [out=255,in=180,blue] ($(mid)+(165:1.15)$) to (x6);
    	
	\draw[out=225,in=180,blue] (x6) to (x7);	
    \draw[out=210,in=135,blue] (x7) to ($(mid)+(225:1.18)$)
    [out=315,in=240,blue] ($(mid)+(225:1.18)$) to (x8);	
    \draw[out=240,in=165,blue] (x8) to ($(mid)+(255:1.35)$)
    [out=345,in=270,blue] ($(mid)+(255:1.35)$) to (x9); 
    
	\draw[out=315,in=270,blue] (x9) to (x10);	
    \draw[out=300,in=225,blue] (x10) to ($(mid)+(315:1.4)$)
    [out=45,in=330,blue] ($(mid)+(315:1.4)$) to (x11);   	
    	
%

\end{tikzpicture}

\caption{Step 3, with resulting graph $H_3$ depicted.}\label{fig:stage3}
\end{figure}
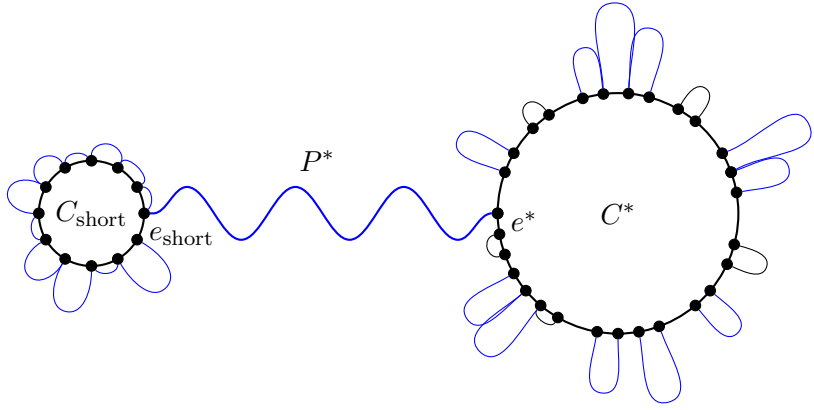

\item Construct a Hamilton cycle by connecting the vertex of $e^*$ and the vertex of $e_{\text{short}}$ that are not connected by $P^*$ by a path $P$ in $G_4$, whose internal vertices are exactly $V(G)\setminus V(H_3)$. Denote
$$
C_H \coloneqq H_3\cup P \setminus \big( \{e_0,...,e_K,e^*,e_{\text{short}}\}\cup \{ e_{i,j} \}_{i,j} \big).
$$
Then the constructed $H_4 \coloneqq H_3 \cup P$ contains the Hamilton cycle $C_H$ and $K+b\cdot t+3$ additional edges, and all cycle lengths in $\left[ (d+b+1)\cdot t+1 , \ell^*+L \right] \cup \left[ n-L,n \right]$.

\begin{figure}[h]
\centering
\begin{tikzpicture}[vertex/.style={draw,circle,color=black,fill=black,inner sep=1,minimum width=4pt},scale=1]
    
    \draw[thick,blue] (0,2.5) arc (90:240:2.5);
	\draw[thick] (0,2.5) arc (90:-120:2.5);
	
	\node[vertex] (x) at (0,2.5) {};
    \node[vertex] (y) at (330:2.5) {};
    \node[vertex] (z) at (290:2.5) {};
    \node[vertex] (w) at (240:2.5) {};
    
    \draw[out=270,in=150] (x) to (y);
    \draw[out=135,in=35] (z) to (w);
    
    \node at (-3,0) {$P$};
    \node at (0.3,0) {$e ^*$};
    \node at (1.2,0.7) {$C^*$};
    \node at (0,-1.6) {$e_{\text{short}}$};
    \node at (2,-2.1) {$P^*$};
    
    \draw[out=265,in=260] (85:2.5) to (80:2.5);	
    \draw[out=260,in=237] (80:2.5) to (67:2.5);
    
    \draw[out=240,in=231] (60:2.5) to (51:2.5);
    
    \draw[out=226,in=202] (46:2.5) to (22:2.5);
    \draw[out=202,in=197] (22:2.5) to (17:2.5);
    \draw[out=197,in=187] (17:2.5) to (7:2.5);
    
    \draw[out=180,in=172] (0:2.5) to (352:2.5);
    \draw[out=172,in=167] (352:2.5) to (347:2.5);
    
    \draw[out=162,in=152] (342:2.5) to (332:2.5);
    
    \draw[out=110,in=105] (290:2.5) to (285:2.5);
    \draw[out=105,in=98] (285:2.5) to (278:2.5);
    \draw[out=98,in=90] (278:2.5) to (270:2.5);
    \draw[out=90,in=80] (270:2.5) to (260:2.5);
    \draw[out=80,in=75] (260:2.5) to (255:2.5);
    \draw[out=75,in=65] (255:2.5) to (245:2.5);
    \draw[out=65,in=60] (245:2.5) to (240:2.5);


\end{tikzpicture}
\qquad \qquad
\begin{tikzpicture}[vertex/.style={draw,circle,color=black,fill=black,inner sep=1,minimum width=4pt},scale=1]
    
    \draw[out=110,in=105,lightgray] (290:2.5) to (285:2.5);
    \draw[out=105,in=98,lightgray] (285:2.5) to (278:2.5);
    \draw[out=98,in=90,lightgray] (278:2.5) to (270:2.5);
    \draw[out=90,in=80,lightgray] (270:2.5) to (260:2.5);
    \draw[out=80,in=75,lightgray] (260:2.5) to (255:2.5);
    \draw[out=75,in=65,lightgray] (255:2.5) to (245:2.5);
    \draw[out=65,in=60,lightgray] (245:2.5) to (240:2.5);
    
	\node[vertex,color=lightgray] (z) at (290:2.5) {};
    \node[vertex,color=lightgray] (w) at (240:2.5) {};    
    
    \draw[thick] (0,2.5) arc (90:240:2.5);
	\draw[thick] (0,2.5) arc (90:-120:2.5);
	
	\node[vertex] (x) at (0,2.5) {};
    \node[vertex] (y) at (330:2.5) {};

    \draw[out=270,in=150] (x) to (y);
    \draw[out=135,in=35,lightgray] (z) to (w);

    \draw[out=265,in=260] (85:2.5) to (80:2.5);	
    \draw[out=260,in=237] (80:2.5) to (67:2.5);
    
    \draw[out=240,in=231] (60:2.5) to (51:2.5);
    
    \draw[out=226,in=202] (46:2.5) to (22:2.5);
    \draw[out=202,in=197] (22:2.5) to (17:2.5);
    \draw[out=197,in=187] (17:2.5) to (7:2.5);
    
    \draw[out=180,in=172] (0:2.5) to (352:2.5);
    \draw[out=172,in=167] (352:2.5) to (347:2.5);
    
    \draw[out=162,in=152] (342:2.5) to (332:2.5);

	\node[vertex] (v1) at (-0.7,2.38) {};
    \node[vertex] (v2) at (2,-1.5) {};
	\draw[thick,blue,out=270,in=150] (v1) to (v2);
	\node at (1.5,-1.5) {$f_3$};
	
	\node[vertex] (u1) at (-1.1,2.25) {};
    \node[vertex] (u2) at (1.5,-2) {};
	\draw[thick,blue,out=270,in=150] (u1) to (u2);
	\node at (0.6,-1.8) {$f_{4}$};
	
	\node[vertex] (w1) at (-2.5,0) {};
    \node[vertex] (w2) at (-1.75,-1.8) {};
	\draw[thick,blue,out=310,in=100] (w1) to (w2);
	\node at (-1.5,-1.4) {$f_m$};
	
	\node[vertex] (z1) at (-2,1.47) {};
    \node[vertex] (z2) at (-2.1,-1.4) {};
	\draw[thick,blue,out=280,in=75] (z1) to (z2);
	\node at (-1.48,1) {$f_{m-1}$};
	
	\draw[ultra thick,loosely dotted] (-1.5,-0.5) to (-0.7,-0.25);
	
\end{tikzpicture}

\caption{Steps 4 and 5, with resulting graph $H_4$ (left) and $H_5$ (right) depicted.}\label{fig:stage4}
\end{figure}
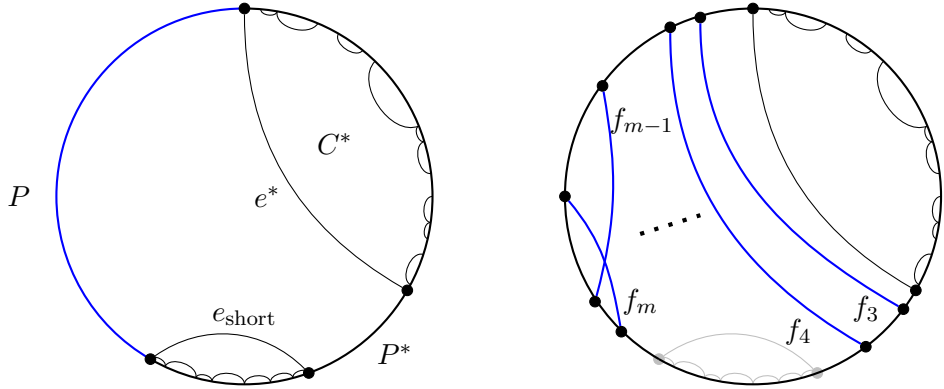

\item Let $m\coloneqq \lfloor n\cdot 2^{-K} \rfloor = o(\log n)$. For $3 \le i \le m$ find an $\ell _i ^*$-shortcut $f_i$ in $G_5$, where $\ell _i^*$ is an integer such that $\left| \ell _i^* -i\cdot 2^K \right| \le n^{0.9}$, and such that the $(\ell _i^*+1)$-path accompanying $f_i$ contains $V\left( C^* \cup \bigcup _{i=0}^KC_i \right)$. We now have that  $H_4 \cup \{f_i\}$ contains all cycle lengths in $\left[ \ell_i^*+2-L, \ell_i^*+2 \right]$. Since $\ell _i^* \ge \ell_{i+1}^*-L$ for all $i$, and $\ell ^*+L \ge \ell_3^*+2-L,\ \ell _m^* \ge n-L$, we get that $H_4\cup \{f_3,...,f_m\}$ contains all cycle lengths in $\left[ (d+b+1)\cdot t+1,n \right]$.

Finally, add the remaining at most $(d+b+1)\cdot t=o(\log n)$ cycle lengths by finding in $G_5$ an edge $g_{\ell}$ that constitutes an $(\ell-2)$-shortcut with respect to $C_H$, for every $\ell \in [3,(d+b+1)\cdot t]$.

This step adds at most $m+(d+b+1)\cdot t$ edges to the constructed subgraph $H_5 \coloneqq H_4 \cup \{f_3,...,f_m\} \cup \{g_3,...,g_{(d+b+1)t}\}$.
\end{enumerate}

Observe that the resulting subgraph $H_5$ is a union of the Hamilton cycle $C_H$ and an additional set of edges
$$
\{e_0,...,e_K,e^*,e_{\text{short}}\}  \cup \{ e_{i,j} \mid 0\le i \le t-1,0\le j \le b-1 \} \cup \{f_3,...,f_m\} \cup \{g_3,...,g_{(d+b+1)t}\} .
$$
Therefore $H_5$ contains at most
$$
n+K+b\cdot t+m+(d+b+1)\cdot t = n+(1+o(1))\cdot \log _2n
$$
edges. In Section \ref{sec:construction} we prove that the construction of $H_5$ we described is possible with high probability in $G(n,p^*)$. In Section \ref{sec:main} we prove that $H_5$, if it exists as a subgraph of $G$, is indeed pancyclic.

\section{Finding the subgraph in $G(n,p)$} \label{sec:construction}

We follow the steps described in Section \ref{sec:outline}, and show that, in each step, the desired substructure of the respective random graph $G_i$, $i=1,2,3,4,5$, exists with high probability.

We will denote the subgraph output by the $i$'th step of the construction (if successful) by $H_i$.

\subsection*{Step 1}

Recall the notation
$$
K_0 \coloneqq \lfloor \log _2 \left( \frac{\log n}{6\log \log \log n} \right) \rfloor,\ \ell_i \coloneqq 2^i+1
$$
and
$$
b \coloneqq \lceil \log \log n \rceil,\ t\coloneqq \lceil \log _b \log n \rceil.
$$

By Theorem \ref{thm:luczak}, with high probability, $G_1\sim G\left( n, p_1 \right)$ (where $p_1 =\frac{2\log \log n}{n}$) contains a sequence of cycles $C_0,C_1,...,C_{K_0},C_{\text{short}}$ of respective lengths $\ell _0+1,...,\ell_{K_0}+1,b\cdot t+1$.

The following lemma implies that these cycles are also typically vertex disjoint.

\begin{lem}
With high probability, no two cycles of length at most $\ell _{K_0}+1$ in $G_1$ intersect each other.
\end{lem}

\begin{proof}
Using the union bound we can show that, with high probability, $G_1$ does not contain a subgraph with at most $2\ell _{K_0}+1 \le \frac{\log n}{2\log \log \log n}$ vertices and more edges than vertices, which implies the lemma. Indeed, the probability that such a subgraph exists is at most
$$
\sum _{k=4}^{2\ell _{K_0}+1} \binom{n}{k}\cdot \binom{\binom{k}{2}}{k+1} \cdot {p_1}^{k+1} \le \sum _{k=4}^{2\ell _{K_0}+1} (e^2np_1)^k\cdot k \cdot p_1 \le \log ^2n \cdot (2e^2\log \log n)^{\frac{\log n}{2\log \log \log n}}\cdot p_1 = o(1).
$$
\end{proof}

\subsection*{Step 2}

Recall that $G_2 \sim G(n,p_2)$, where $p_2 = \frac{50\log n}{n\cdot \log \log n}$. For each $0\le i \le K_0$ let $\{s_i,t_i\}\coloneqq e_i\in E(C_i)$ be an arbitrary edge of $C_i$.

Recall that $\beta \coloneqq \frac{2(\log \log n) ^2}{\log n}$ and $d = \lfloor \log _{(5\beta)^{-1}}(n/200) \rfloor = (1+o(1))\cdot \frac{\log n}{\log \log n}$.

\begin{lem} \label{lemma:shortpaths}
With high probability $G_2$ contains paths $Q_0,...,Q_{K_0}$ such that
\begin{enumerate}
	\item $Q_i$ is a path between $t_i$ and $s_{i+1}$ for $0\le i \le K_0-1$, and $Q_{K_0}$ is between $t_{K_0}$ and $s_0$;
	\item $Q_0,...,Q_{K_0}$ all have length $d+2$;
	\item $Q_0,...,Q_{K_0}$ are vertex disjoint, and are internally vertex disjoint from $V(H_1)$.
\end{enumerate}
\end{lem}

Recall that $K= \lfloor \log _2 \left( \frac{n}{\sqrt{\log n}} \right) \rfloor$ and $L=2^{K+1}-1$. Before getting to the proof of Lemma \ref{lemma:shortpaths}, we show the following claim.

\begin{claim}\label{clm:G2expander}
With high probability, for every vertex subset $U\subseteq V(G)$ with $|U| \ge n-2L$, there is a vertex subset $U^* \subseteq U$, with $|U^*|\ge (1-\beta )\cdot n$, such that the induced subgraph $G_2[U^*]$ is a $\left( \beta n, 1/3\beta \right)$-expander.
\end{claim}

\begin{proof}
First, observe that, with high probability, for every $U,W\subseteq V(G)$ disjoint subsets with $|U|=|W|=\beta n$ there is an edge in $G_2$ between $U$ and $W$. Indeed, the probability that there are such subsets with no edge between them is at most
$$
\binom{n}{\beta n}^2 \cdot (1-p_2)^{\beta ^2 n^2} \le \left( \frac{e^2}{\beta ^2} \cdot \exp (-\beta np_2)\right) ^{\beta n} \le \left( \frac{\log ^2n}{\log \log n} \cdot \exp \left( -100\log \log n \right) \right) ^{\omega (1)} = o(1).
$$

Now, assume that $G_2$ has the aforementioned property. We reiterate an argument from \cite{FRIED} and show that, in this case, for every such $U$ there is $U^*\subseteq U$ with the desired properties.

For a given $U$, construct $U^*$ as follows. Set $U_0=U$. For $i\ge 0$, if  $|U_i| \ge (1-\beta )n$ and there is $W_i\subseteq U_i$ with $|W_i|\le \beta n$ and $|N_{G_2[U_i]}(W_i)|\le \frac{1}{3\beta}|W_i|$, set $U_{i+1}=U_i\setminus W_i$. Otherwise, terminate the process with $U^*=U_i$. 

Clearly, either the resulting $G_2[U^*]$ is a $\left( \beta n, 1/3\beta \right)$-expander, or $|U^*| < (1-\beta )\cdot n$. In fact, in the latter case, it must be that $(1-2\beta )\cdot n \le |U^*| < (1-\beta )\cdot n$, since at most $\beta n$ vertices are removed in every step of the process. Suppose that this is the case, and denote $W \coloneqq U \setminus U^*$. Then $|N_{G_2}(W)| \le \frac{1}{3\beta}\cdot |W| +|V(G)\setminus U|  \le \frac{2}{3}n$. We therefore have that $W$ and $V(G)\setminus N_{G_2}(W)$ are subsets of size at least $\beta n$ with no edges between them, a contradiction to our assumption.
\end{proof}

With Claim \ref{clm:G2expander} at hand we are now able to prove Lemma \ref{lemma:shortpaths} by appealing to Theorem \ref{thm:hax}.

\begin{proof}[Proof of Lemma \ref{lemma:shortpaths}]

Assume that $G_2$ has the property in the assertion of Claim \ref{clm:G2expander}, and suppose that $Q_0,...,Q_{i-1}$ have already been constructed. We attempt to construct $Q_i$.

Let $U\coloneqq V(G) \setminus \left( V(H_1) \cup \bigcup _{j=0}^{i-1}V(Q_j) \right)$, so that
$$
|U| \ge n-2b\cdot t-2^{K_0+1}-K_0 \cdot (d+2) \ge n-2L,
$$
and let $U^*\subseteq U$ be a subset of size at least $(1-\beta )\cdot n$ such that $G_2[U^*]$ is a $\left( \beta n, 1/3\beta \right)$-expander. Observe that $G_2[U^*]$ satisfies the conditions of Theorem \ref{thm:hax} for $\Delta =\frac{1}{4\beta},N=\frac{1}{50}n,k=\frac{1}{2}\beta n$.

Observe that, at this point, the edges of $G_2$ between $\{s_{i+1},t_i\}$ ($s_0$ in the case $i=K_0$) and $U^*$ have not been sampled yet. The probability that $s_{i+1}$ does not have a neighbour in $U^*$ is at most $(1-p_2)^{(1-\beta )n} =o({K_0}^{-1})$. Assume that there is such a neighbour, say $u$. By Theorem \ref{thm:hax}, $G_2[U^*]$ contains a complete $\frac{1}{5\beta}$-ary tree of depth $d$ rooted in $u$. This tree has at least $\frac{\beta}{40}\cdot n$ leaves. The probability that none of these leaves is a neighbour of $t_{i}$ in $G_2$ is at most
$$
(1-p_2)^{\beta n/40} \le \exp \left( -\frac{1}{40}\cdot \frac{50\log n}{n\cdot \log \log n} \cdot \frac{2(\log \log n)^2}{\log n}\cdot n \right) = o({K_0}^{-1}).
$$
Now, if indeed $t_{i}$ has a neighbour among the tree's leaves, say $w$, the path from $s_{i+1}$ to $u$, down the the tree to $w$, and from $w$ to $t_{i}$ is a path of length $d+2$ that intersects $V(C_{\text{short}})\cup \bigcup _{j=0}^{K_0}V(C_j) \cup \bigcup _{j=0}^{i-1}V(Q_j)$ only in $\{s_{i+1},t_{i} \}$.

Finally, for every $i$ we showed that the probability that such a path $Q_i$ does not exist is at most $o({K_0}^{-1})$, and therefore, by the union bound, a sequence $Q_0,...,Q_{K_0}$ as required exists with high probability.
\end{proof}

Now $\left( \bigcup _{i=0}^{K_0} Q_i \right) \cup \{e_0,...,e_{K_0}\}$ is a cycle, denote it by $C^*$. We have
$$
\ell ^* \coloneqq |C^*| = (K_0 +1)\cdot \left( d+3 \right) = (1+o(1))\cdot \log _2n.
$$

\subsection*{Step 3}

Recall that $G_3 \sim G(n,p_3)$, with $p_3 = \frac{50\log n}{n\cdot \log \log n}$. Let $e_{K_0+1},...,e_{K},e^*$ be distinct edges of $C^*\setminus \{e_0,...,e_{K_0}\}$, such that $e^*$ is vertex disjoint from $e_{K_0+1},...,e_K$, and denote $e_i=\{s_i,t_i\}$ such that $s_i$ is the predecessor of $t_i$ on $C^*$ for all $i$, according to an arbitrary orientation of $C^*$.

As a preparation for a proof that the construction in Step 3 is possible with high probability, observe that $G_3$ and $G_2$ are drawn from the same distribution, and therefore Claim \ref{clm:G2expander} also holds for $G_3$. That is, we have that, with high probability, for every $U\subseteq V(G)$ with $|U| \ge n-2L$, there is $U^* \subseteq U$, with $|U^*|\ge (1-\beta )\cdot n$, such that $G_3[U^*]$ is a $\left( \beta n, 1/3\beta \right)$-expander. In the proofs of the following two lemmas, we assume that indeed $G_3$ has this property.

\begin{lem}\label{lem:step3long}
With high probability $G_3$ contains paths $Q_{K_0+1},...,Q_{K}$ such that
\begin{enumerate}
	\item $Q_i$ is a path between $s_i$ and $t_i$ for $K_0+1\le i \le K$;
	\item $Q_i$ has length $\ell _i+1$ for $K_0+1\le i \le K$;
	\item $Q_{K_0+1},...,Q_{K}$ are internally vertex disjoint from each other and from $V(H_2)$.
\end{enumerate}
\end{lem}

\begin{proof}
Suppose that $Q_{K_0+1},...,Q_{i-1}$ were found, and attempt to construct $Q_i$.

Here, as in Lemma \ref{lemma:shortpaths}, we will appeal to Theorem \ref{thm:hax}.

Let $U\coloneqq V(G) \setminus \left( V(H_2) \cup \bigcup _{j=0}^{i-1}V(Q_j) \right)$ and observe that
$$
|U| \ge n-2^{K+1}-K_0 \cdot (d+2) \ge n-2L.
$$
Let $U^*\subseteq U$ be a subset with at least $(1-\beta )\cdot n$ vertices such that $G_3[U^*]$ is a $\left( \beta n, 1/3\beta \right)$-expander.

As in the proof of Lemma \ref{lemma:shortpaths}, $G_3[U^*]$ satisfies the conditions of Theorem \ref{thm:hax} for the same parameters $\Delta =\frac{1}{4\beta},N=\frac{1}{50}n,k=\frac{1}{2}\beta n$. Recall that $d = \lfloor \log _{(5\beta)^{-1}}(n/200) \rfloor $, and let $T$ be the tree consisting of two complete $\frac{1}{5\beta}$-ary trees of depth $d$, whose roots are connected by a path of length $\ell _i-2d-1$ (which is positive for $i>K_0$). By Theorem \ref{thm:hax}, $U^*$ contains a copy of $T$ (rooted at an arbitrary vertex).

Let $L_{s_i},L_{t_i}\subseteq U^*$ be the sets of leaves of the embedding of $T$ that correspond to the first and the second subtrees of $T$ that are connected by a path. By the definition of $T$ we have that $|L_{s_i}|=|L_{t_i}|\ge \frac{\beta}{40}n$. Observe that $s_i$ and $t_i$ each belong to at most one other path among $Q_{K_0+1},...,Q_{i-1}$. For $v\in \{s_i,t_i\}$ do the following. If $v\notin V(Q_j)$ for $K_0<j\le i-1$, then choose an arbitrary subset of $L_v$ of size $\frac{1}{2}|L_v|$ and connect $v$ to one of the vertices in the subset by an edge from $E(G_3)$, if there is a neighbour of $v$ in the subset. If $v\in V(Q_j)$ for some $K_0<j\le i-1$, connect $v$ to a vertex of $L_v$ by a previously unexposed edge from $E(G_3)$, if such an edge exists. In both cases, at least $\frac{1}{2}|L_v| \ge \frac{\beta}{80}n$ edges are considered. Therefore, the probability that there is no edge between $v$ and (the subset of) $L_v$ is at most
$$
(1-p_3)^{\beta n/80} \le \exp \left( -\frac{1}{80}\cdot \frac{50\log n}{n\cdot \log \log n} \cdot \frac{2(\log \log n)^2}{\log n}\cdot n \right) = \exp\left( -\frac{5}{4}\log \log n \right) = o({K}^{-1}).
$$
In the case that an edge is found, denote it by $e_v$.

Now, $e_{s_i},e_{t_i}$ along with the path of length $\ell _i-1$ in $T$ between the two leaves connected to $s_i$ and $t_i$ constitute a path between $s_i$ and $t_i$ of length $\ell_i+1$, which is internally contained in $U$, and therefore, by the definition of $U$, is internally vertex disjoint from $V(H_2),Q_{K_0},...,Q_{i-1}$. Call this path $Q_i$.

The probability that there is $K_0+1\le i \le K$ for which we did not manage to find a path $Q_i$ in this way is at most the probability that $G_3$ does not have the property in the assertion of Claim \ref{clm:G2expander}, or $s_i$ or $t_i$ did not have a leaf neighbour in the embedding of $T$ for some $i$, both of which are of order $o(1)$.
\end{proof}

For $K_0+1\le i \le K$, denote by $C_i$ the cycle $Q_i \cup \{e_i\}$.

Let $v_1,...,v_{bt+1}$ be the vertices of $C_{\text{short}}$ according to their order on the cycle, let $\sigma :  \{0,1,...,t-1\} \times \{0,1,...,b-1\}\to [tb]$ be a bijection and denote $e_{i,j} = \{v_{\sigma (i,j)},v_{\sigma (i,j) +1}\}$ and $e_{\text{short}}=\{v_1,v_{bt+1}\}$.


\begin{lem}\label{lem:step3short}
With high probability $G_3$ contains paths $\{P_{i,j} \mid 0\le i\le t-1,0\le j\le b-1\}$ such that
\begin{enumerate}
	\item $P_{i,j}$ is a path between $v_{\sigma (i,j)}$ and $v_{\sigma (i,j)+1}$, for all $i$ and $j$;
	\item $P_{i,j}$ has length $d+2+j\cdot b^i$, for all $i$ and $j$;
	\item $\{P_{i,j} \mid 0\le i\le t-1,0\le j\le b-1\}$ are internally vertex disjoint from each other and from $V(H_2)\cup \bigcup _{i=K_0+1}^KV(C_i)$.
\end{enumerate}
\end{lem}

\begin{proof}
The proof follows similar steps to the proofs of Lemma \ref{lemma:shortpaths} and Lemma \ref{lem:step3long} by appealing to Theorem \ref{thm:hax}. Assume that $P_{\sigma ^{-1}(1)},...,P_{\sigma ^{-1}(k-1)}$ have already been found, and let $(i,j) = \sigma ^{-1}(k)$. Let $U^* \subseteq U\coloneqq V(G) \setminus \left( V(H_2) \cup \bigcup _{r=K_0+1}^{K}V(C_r) \cup \bigcup _{r=1}^{k-1} V(P_{\sigma ^{-1}(r)})\right)$ be such that $|U^*| \ge (1-\beta )\cdot n$ and $G_3[U^*]$ is a $\left( \beta n, 1/3\beta \right)$-expander.

Let $s_{k+1}$ be a neighbour of $v_{k+1}$ from among the first $\frac{n}{\log \log n}$ vertices of $U^*$. The edges between $v_{k+1}$ and $U^*$ in $G_3$ have not been sampled yet, and the probability that no such neighbour exists is at most $(1-p_3)^{\frac{n}{\log \log n}}=o(1/tb)$.

As in the previous proofs, by Claim \ref{clm:G2expander} and Theorem \ref{thm:hax} we have that $G_3[U^*]$ contains a tree which consists of a complete $\frac{1}{5\beta}$-ary tree of depth $d$ with a path of length $j\cdot b^i$ (this can possibly be 0, in which case the path is just a vertex) attached to its root, and such that the other end of the path is $s_{k+1}$. Let $L_k$ be the set of leaves in the tree, so that $|L_k| \ge \frac{\beta}{40}n$. At most $\frac{n}{\log \log n}$ of the leaves were considered as neighbours of $v_{k}$ in previous steps, and therefore the probability that $v_k$ does not have a neighbour $t_k$ from among the remaining leaves is at most $(1-p_3)^{\frac{\beta}{50}n}=o(1/tb)$. Now the path from $v_{k+1}$, through $s_{k+1}$, along the $(jb^i)$-path, down the tree to $t_k$ and then to $v_k$, satisfies all the requirements to be $P_{i,j}$.

The probability that for some $k$ one of the vertices $s_{k+1},t_k$ was not found is of order $o(1)$, and therefore, with high probability, this construction ends successfully.
\end{proof}

We remain with finding a path between $e_{\text{short}}$ and $e^*$, which is done in the following lemma.

\begin{lem}\label{lem:step3path}
With high probability $G_3$ contains a path $P^*$ of length $d+2$ between a vertex of $e_{\text{short}}$ and a vertex of $e^*$, which is internally disjoint from $V(H_2) \cup \bigcup _{i=K_0+1}^{K}V(C_i) \cup \bigcup _{i,j} V(P_{i,j})$.
\end{lem}

\begin{proof}
As in previous constructions in this step, let $U=V(G)\setminus \left( V(H_2) \cup \bigcup _{i=K_0+1}^{K}V(C_i) \cup \bigcup _{i,j} V(P_{i,j}) \right)$ and let $U^*$ be a large subset spanning an expander. At least $\frac{1}{2}n$ vertices of $U^*$ have not yet been considered as neighbours of $v_{bt+1}$, and with high probability at least one of them is, denote it by $s$. By Claim \ref{clm:G2expander} and Theorem \ref{thm:hax}, $G_3[U^*]$ contains a complete $\frac{1}{5\beta}$-ary tree of depth $d$ rooted in $s$. As none of the edges in $G_3$ of the vertices of $e^*$ have been sampled yet, with high probability there is an edge between one of them and the tree's at least $\frac{\beta}{40}n$ leaves, which together with the path from the leaf to $s$ and with $\{s,v_{bt+1}\}$ forms a path $P^*$ satisfying the conditions.
\end{proof}

\subsection*{Step 4}

Denote $X = V(G) \setminus V(H_3)$, and let $s_H\in e^*,t_H\in e_{\text{short}}$ be the vertices of $e^*,e_{\text{short}}$ not already connected by $P^*$.

\begin{claim}
With high probability there is a path $P$ in $G_4$ between $s_H$ and $t_H$, whose vertex set is $V(P)=X\cup \{s_H,t_H\}$.
\end{claim}

\begin{proof}
Consider the induced subgraph $G_4[X]\sim G(|X|,p_4)$. We have
\begin{eqnarray*}
|X|\cdot p_4 & \ge & (n-2L)\cdot \left( \frac{\log n + 10\sqrt{\log n}}{n} \right) \\
& \ge & \log n \cdot \left( 1-\frac{4}{\sqrt{\log n}} \right)\cdot \left( 1+\frac{10}{\sqrt{\log n}} \right) \\
&=& \log n + \omega (\log \log n)\\
&=& \log |X| + \omega (\log \log |X|)
\end{eqnarray*}

Therefore, by Lemma \ref{thm:kriv}, there is a set $S\subseteq X$ with $|S| = \frac{1}{4}|X| \ge \frac{1}{5}n$, and for every $s\in S$ there is a subset $T_s\subseteq X$ with $|T_s| \ge \frac{1}{4}|X| \ge \frac{1}{5}n$, such that there is a Hamilton path in $G_4[X]$ between $s$ and $t$ for every $t\in T_s$.

The set $E_{G_4}(s_H,X)$ has not yet been sampled. The probability that $s_H$ has no neighbour in $S$ is at most $(1-p_4)^{n/5}=o(1)$. Assume that there is one, and denote it by $s$. Similarly, the probability that $t_H$ has no neighbour in $T_s$ is at most $(1-p_4)^{n/5}=o(1)$, denote such a neighbour by $t$. Now the Hamilton path in $G_4[X]$ between $s$ and $t$, along with the edges $\{s_H,s\},\{t_H,t\}$, constitute a path $P$ with $V(P)=X\cup \{s_H,t_H\}$, as desired.
\end{proof}

Denote the obtained Hamilton cycle $H_3\cup P \setminus \left( \{e_0,...,e_K,e^*,e_{\text{short}}\}\cup \{ e_{i,j} \}_{i,j} \right)$ by $C_H$.

\subsection*{Step 5}

Let $m \coloneqq \lfloor n\cdot 2^{-K} \rfloor$.

\begin{lem}
With high probability $G_5$ contains edges $f_3,...,f_m$, such that the following hold for every $i$.
\begin{enumerate}
	\item There is $\ell _i^* \in \left[ i\cdot 2^K - n^{0.9}, i\cdot 2^K +n^{0.9} \right]$ such that $f_i$ is an $\ell _i*$-shortcut with respect to $C_H$;
	\item The $(\ell _i^* +1)$-path on $C_H$ that connects the vertices of $f_i$ contains $V\left( C^* \cup \bigcup _{i=0}^KC_i \right)$.
\end{enumerate}
\end{lem}

\begin{proof}
For every $3\le i \le m$ there is a set of at least $\frac{1}{5}\cdot n^{1.8}$ potential edges that satisfy the conditions. The probability that there is $3\le i\le m$ for which none of these edges appears in $G_5$ is at most
$$
m\cdot (1-p_5)^{n^{1.8}/5} = o(1). 
$$
\end{proof}

\begin{lem}
With high probability $G_5$ contains an $(\ell -2)$-shortcut with respect to $C_H$ for every $\ell \in [3,(d+b) \cdot t]$.
\end{lem}

\begin{proof}
For a given $\ell \in [3,(d+b)\cdot t]$, the probability that such an $(\ell -2)$-shortcut does not exist is $(1-p_5)^n=o(1/\log n)$, and by the union bound we obtain the lemma, as $(d+b)t=o(\log n)$.
\end{proof}

\section{Proof of Theorem \ref{thm:main}} \label{sec:main}

The following lemma, referring to the subgraph $H_5\subseteq G$ described in Section \ref{sec:outline} and whose construction is shown to be possible with high probability in Section \ref{sec:construction}, completes the proof of Theorem \ref{thm:main}.

\begin{lem}
The subgraph $H_5$ is pancyclic.
\end{lem}

\begin{proof}
Let $\ell \in [3,n]$. We show that $H_5$ contains a cycle of length $\ell$. We divide the proof into cases based on a subinterval of $[3,n]$ that $\ell$ resides in. The subintervals are covering $[3,n]$ but not necessarily disjoint, so $\ell$ may be covered by more than one subinterval.

\begin{itemize}
\item If $\ell \in [3,(d+b+1)\cdot t]$, then $g_{\ell}$ is an $(\ell -2)$-shortcut with respect to $C_H$, so that $g_{\ell}$ and its accompanying $(\ell-1)$-path form an $\ell$-cycle.

\item If $\ell \in [(d+b+1)\cdot t+1,(d+b+1)\cdot t+b^t]$, let $k=\ell -(d+b+1)\cdot t-1$, so $0\le k\le b^t-1$ can be encoded in base $b$ using $t$ digits. Let $(k_{t-1},k_{t-2},...,k_1,k_0)$ be its encoding, that is, $0\le k_i \le b-1$ for all $i$, and $k=\sum _{i=0}^{t-1}k_ib^i$. Then
$$
\{e_{\text{short}}\} \cup \bigcup _{j=k_i}P_{i,j} \cup \bigcup _{j\neq k_i}\{ e_{i,j}\}
$$
is a cycle of length $\ell$ in $H_5$. Indeed, it is a cycle, since it is the result of replacing a subset of the edges of $C_{\text{short}}$ with internally disjoint paths, and its length is
$$
1+(b-1)\cdot t+\sum _{i=0}^{t-1}e(P_{i,k_1}) = 1+(b-1)\cdot t+\sum _{i=0}^{t-1}\left( d+2+k_ib^i \right) =1+(d+b+1)\cdot t+k=\ell.
$$

\item If $\ell \in [\ell ^*,\ell^*+L]$, let $k=\ell-\ell^*$, so $0\le k \le 2^{K+1}-1$ can be encoded by $K+1$ binary digits, say $k=\sum_{i=0}^{K} k_i2^i$, where $k_i\in \{0,1\}$. Then
$$
\left( C^* \cup \bigcup _{i:k_i=1}C_i \right) \setminus \{ e_i \mid k_i=1\}
$$
is a cycle in $H_5$ with length $\ell$. It is a cycle because it is the result of replacing edges of $C^*$ with internally disjoint paths. The length is indeed
$$
\ell^*+\sum _{i=0}^Kk_i\cdot (e(C_i)-1) = \ell^*+\sum_{i=0}^{K} k_i2^i=\ell^*+k=\ell.
$$

\item If $\ell \in \left[ \left(\frac{1}{2}i-\frac{4}{5}\right) \cdot L,\left(\frac{1}{2}i-\frac{1}{5}\right) \cdot L \right]$, where $3\le i \le m$, then in particular $\ell \in \left[ \ell_i^*+2-L, \ell_i^*+2 \right]$. Similarly to the previous case, let $\ell_i^*+2-\ell =k=\sum_{i=0}^{K} k_i2^i$, where $k_i\in \{0,1\}$. Denote by $C_i^*$ the cycle of length $\ell_i^*+2$ comprised of $f_i$ and its accompanying $(\ell_i^*+1)$-path. Then
$$
\left( C_i^* \setminus \bigcup _{i:k_i=1}C_i \right) \cup \{ e_i \mid k_i=1\}
$$
is a cycle of length $\ell^*+2-k=\ell$.

\item If $\ell \in [n-L,n]$ then for $n-\ell =k=\sum_{i=0}^{K} k_i2^i,\ k_i\in \{0,1\}$, we get a cycle
$$
\left( C_H \setminus \bigcup _{i:k_i=1}C_i \right) \cup \{ e_i \mid k_i=1\}
$$
of length $n-k=\ell$.
\end{itemize}

Observe that
\begin{align*}
(d+b+1)\cdot t + b^t  &\ge  b^{\log _b\log n} = \log n \ge \ell ^*\ ;\\
\ell ^* + L  &\ge  \left( \frac{1}{2}\cdot 3-\frac{4}{5} \right) \cdot L\ ;\\
\left( \frac{1}{2}\cdot i-\frac{1}{5} \right) \cdot L  &\ge  \left( \frac{1}{2}\cdot (i+1)-\frac{4}{5} \right) \cdot L ;\ \\
\left( \frac{1}{2}\cdot m-\frac{1}{5} \right) \cdot L  &\ge \frac{1}{2}(n\cdot 2^{-K}-1)\cdot (2^{K+1}-1)-\frac{1}{5}L \ge n-\frac{1}{2}L-\frac{1}{5}L-O(\sqrt{\log n}) \ge n-L\ ;
\end{align*}
and therefore the subintervals indeed cover $[3,n]$, and so $H_5$ is pancyclic.
\end{proof}

\end{document}